\documentclass[11pt, a4paper]{article}
\usepackage{fullpage}
\usepackage[utf8]{inputenc}
\usepackage{amsmath, amssymb, amsfonts, amsthm, mathtools, xcolor}
\usepackage{authblk}

\newcommand{\setbuilder}[2]{\left\{#1\ \colon #2\right\}}
\DeclareMathOperator{\Span}{span}

\newcommand{\R}{{\mathbb R}}

\newcommand{\Sph}{{\mathbb S}}
\newcommand{\compl}[1]{\overline{#1}}
\newcommand{\epsi}{\varepsilon}

\theoremstyle{plain}

\newtheorem{problem}{Problem}
\newtheorem{theorem}{Theorem}
\newtheorem{corollary}[theorem]{Corollary}
\newtheorem{proposition}[theorem]{Proposition}
\newtheorem{lemma}[theorem]{Lemma}

\theoremstyle{definition}
\newtheorem{definition}{Definition}

\title{Embedding graphs in Euclidean space\footnote{An extended abstract of this paper was published in EUROCOMB'17, Vienna, 28~August--1~September 2017, Electronic Notes in Discrete Mathematics, \textbf{61} (2017), 475--481.}}

\author[1]{N\'ora Frankl\thanks{\texttt{n.frankl@lse.ac.uk} N. Frankl was partially supported by the National Research, Development, and Innovation Office, NKFIH Grant K119670.}}
\author[2]{Andrey Kupavskii\thanks{\texttt{kupavskii@yandex.ru} . A. Kupavskii was partially supported by the Swiss National Science Foundation grants no. 200020-162884 and 200021-175977 and  by the EPSRC grant no. EP/N019504/1.}}
\author[1]{Konrad J. Swanepoel\thanks{\texttt{k.swanepoel@lse.ac.uk}}}
\affil[1]{Department of Mathematics, London School of Economics and Political Science, London}
\affil[2]{Moscow Institute of Physics and Technology, University of Birmingham}
\date{}

\begin{document}
\maketitle

\begin{abstract}
The dimension of a graph $G$ is the smallest $d$ for which its vertices can be embedded in $d$-dimensional Euclidean space in the sense that the distances between endpoints of edges equal $1$ (but there may be other unit distances).
Answering a question of Erd\H{o}s and Simonovits [Ars Combin.\ \textbf{9} (1980) 229--246], we show that any graph with less than $\binom{d+2}{2}$ edges has dimension at most $d$.
Improving their result, we prove that that the dimension of a graph with maximum degree $d$ is at most $d$.
We show the following Ramsey result: if each edge of the complete graph on $2d$ vertices is coloured red or blue, then either the red graph or the blue graph can be embedded in Euclidean $d$-space.
We also derive analogous results for embeddings of graphs into the $(d-1)$-dimensional sphere of radius $1/\sqrt{2}$.
\end{abstract}


\section{Introduction}

\begin{definition} A graph $G=(V,E)$ is a \emph{unit distance graph} in Euclidean space $\R^d$, if $V\subset \R^d$ and \[E\subseteq \setbuilder{(x,y)}{x,y\in V, |x-y|=1}.\]
(Note that we do not require the edge set of a unit distance graph to contain all unit-distance pairs.)
We say that a graph $G$ is \emph{realizable} in a subset $X$ of $\R^d$, if there exists a unit distance graph $G'$ in $\R^d$ on a set of vertices $X_0\subset X$, which is isomorphic to $G$. We will use this notion for $X = \R^d$ and for $X = \Sph^{d-1}$, where $\Sph^{d-1}$ is the sphere of radius $1/\sqrt 2$ with center in the origin.
\end{definition}

Erd\H{o}s, Harary and Tutte \cite{EHT65} introduced the concept of Euclidean dimension $\dim G$ of a graph $G$.

\begin{definition}The \emph{Euclidean dimension} $\dim G$  (\emph{spherical dimension} $\dim_S G$) of a graph $G$ is equal to the smallest integer $k$ such that $G$ is realizable in $\R^k$ (on $\Sph^{k-1}\subset \R^k$).
\end{definition}

Erd\H{o}s and Simonovits \cite{ES80} showed that if $G$ has maximum degree $d$ then $\dim G\le \dim_S G\le d+2$.
In Theorem~\ref{max degree d} we improve this result.

\begin{theorem}\label{max degree d} Let $d\geq 1$ and let $G=(V,E)$ be a graph with maximum degree $d$. Then $G$ is a unit distance graph in $\R^d$ except if $d=3$ and $G$ contains $K_{3,3}$.
\end{theorem}

We also show the following simple result.

\begin{proposition}\label{max degree d-1} Let $d\geq 2$. Any graph $G=(V,E)$ with maximum degree $d-1$ has spherical dimension at most $d$.
\end{proposition}

\begin{definition} Let $f(d)$ denote the least number for which there is a graph with $f(d)$ edges that is not realizable in $\R^d$.
\end{definition}

There are some natural upper bounds on $f(d)$. Since $K_{d+2}$ is not realizable in $\R^d$, it is clear that $f(d)\le \binom{d+2}{2}$. For $d=3$ we have $f(3)\le9<\binom{3+2}{2}$, since $K_{3,3}$ can not be realized in $\R^3$.
In \cite{ES80}, Erd\H{o}s and Simonovits asked if $f(d)=\binom{d+2}{2}$ for $d>3$.
House \cite{Hou13} proved that $f(3)=9$, and that $K_{3,3}$ is the only graph with $9$ edges that can not be realized in $\R^3$. Chaffee and Noble \cite{CN16} showed that $f(4)=\binom{4+2}{2}=15$, and there are only two graphs, $K_6$ and $K_{3,3,1}$, with $15$ edges that can not be realized in $\R^4$ as a unit distance graph.
Recently, they showed \cite{CN17} that $f(5)=\binom{5+2}{2}=21$, and that $K_7$ is the only graph with $21$ edges that cannot be realized in $\R^5$ as a unit distance graph.
We confirm the above-mentioned question of Erd\H{o}s and Simonovits as part of the following result.

\begin{theorem}\label{d+2 choose 2} Let $d>3$. Any graph $G$ with less than $\binom{d+2}{2}$ edges can be realized in $\R^d$. If $G$ moreover does not contain $K_{d+2}-K_3$ or $K_{d+1}$, then it can be realized in $\Sph^{d-1}$.
\end{theorem}

Ramsey-type questions about unit distance graphs have been studied by Kupavskii, Raigorodskii and Titova \cite{KRT13} and by Alon and Kupavskii \cite{AK14}. In \cite{AK14} the first of the following quantities were introduced.
\begin{definition}
Let $f_D(s)$ denote the smallest possible $d$, such that for any graph $G$ on $s$ vertices, either $G$ or its complement $\compl{G}$ can be realized as a unit distance
graph in $\R^d$.
Similarly, we define $f_{SD}(s)$ to be the smallest possible $d$, such that for any graph $G$ on $s$ vertices, either $G$ or its complement $\compl{G}$ can be realized as a unit distance
graph in $\Sph^{d-1}$.
\end{definition}
In \cite{AK14} it is shown that $f_D(s)=(\frac{1}{2}+o(1))s$.
We determine the exact value of $f_{SD}(s)$ and give almost sharp bounds on $f_D(s)$.

\begin{theorem}\label{Ramsey}
For any $d,s\geq 1$, $f_{SD}(s)=\lceil (s+1)/2 \rceil$ and $\lceil (s-1)/2 \rceil\le f_D(s)\le \lceil s/2 \rceil$.
\end{theorem}

\section{Maximum degree}

We use the following lemma of Lov\'asz in the proofs of the results on bounded maximum degrees.

\begin{lemma}[\cite{Lov66}]\label{Lovasz}
Let $G=(V,E)$ be a graph with maximum degree $k$ and let $k_1,\dots, k_{\alpha}$ be non-negative integers such that $k_1+\dots+k_{\alpha}=k-\alpha+1$. Then there is a partition $V=V_1\cup\dots\cup V_{\alpha}$ of the vertex set into $\alpha$ parts such that the maximum degree in $G[V_i]$ is at most $k_i$, $i=1,\dots, \alpha$.
\end{lemma}

The proof of Proposition~\ref{max degree d-1} is a simple induction.

\begin{proof}[Proof of Proposition~\ref{max degree d-1}]
The proof is by induction on $d$. For $d=2$ and $d=3$ the theorem is easy to verify. Let $V=V_1\cup V_2$ be a partition as in Lemma~\ref{Lovasz} for $\alpha=2$, $k_1=\lfloor\frac{d-2}{2}\rfloor$, and $k_2=\lceil\frac{d-2}{2}\rceil$. Then by the induction hypothesis, $G[V_i]$ can be represented on a $\Sph^{k_i}$ in $\R^{k_i+1}$. Represent $G[V_1]$ and $G[V_2]$ on $\Sph^{k_1}$ and $\Sph^{k_2}$ in orthogonal subspaces of dimension $k_1+1$ and $k_2+1$, respectively.
Since the distance between any point in $\Sph^{k_1}$ and any point in $\Sph^{k_2}$ is $1$, and both spheres are subspheres of $\Sph^{d-1}$, we obtain a representation of $G$ in $\Sph^{d-1}$.
\end{proof}

In the proof of Theorem~\ref{max degree d} we use
Lemma~\ref{more buses}, which is a strengthening of a special case of Lemma~\ref{Lovasz}, and
Proposition~\ref{three}, which gives an embedding of cycles in sufficiently general position on the $2$-sphere.

\begin{lemma}\label{more buses} Let $G=(V,E)$ be a graph with maximum degree $d$.

If $d$ is even, then there is a partition $V=V_1\cup\dots\cup V_{d/2}$ such that the maximum degree of $G[V_i]$ is  at most $1$ for $1\leq i <d/2$, the maximum degree of $G[V_{d/2}]$ is at most $2$, and any $v\in V_{d/2}$ of degree $2$ in $G[V_{d/2}]$ has exactly $2$ neighbours in each $V_i$.

If $d$ is odd, then there is a partition $V=V_1\cup\dots\cup V_{(d-1)/2}$ such that the maximum degree of $G[V_i]$ is  at most $1$ for $1\leq i <(d-3)/2$, the maximum degree of $G[V_{(d-3)/2}]$ and $G[V_{(d-1)/2}]$ is at most $2$, any degree $2$ vertex in $G[V_{(d-3)/2}]$ has exactly $2$ neighbours in each $V_i$ for $i\leq(d-5)/2$ and exactly $3$ neighbours in $V_{(d-1)/2}$, and any degree $2$ vertex of $G[V_{(d-1)/2}]$ has at least $2$ neighbours in each $V_i$ for $i\leq (d-3)/2$ and at most $3$ neighbours in $V_{(d-3)/2}$.
\end{lemma}

\begin{proof}
$d$ is even: Let $V=V_1\cup\dots\cup V_{d/2}$ be a partition for which $\sum_{i=1}^{d/2}{e(G[V_i])}$ is minimal, where $e(G[V_i])$ denotes the number of edges in $G[V_i]$. For such a partition, each $v\in V_i$ is joined to at most 2 vertices in $V_i$, otherwise we could move $v$ into some other part $V_j$ to decrease the sum of the number of edges in the induced graphs. Similarly, any $v\in V_i$ joined to exactly $2$ other vertices in $V_i$ has exactly $2$ neighbours in each $V_j$. Hence we can move each vertex $v\in V_i$ with degree $2$ in $G[V_i]$ one by one to $V_{d/2}$.

$d$ is odd: Let $V=V_1\cup\dots\cup V_{(d-1)/2}$ be a partition for which $\sum_{i=1}^{(d-1)/2}{e(G[V_i])}$ is minimal. Again, for such a partition each $v\in V_i$ is joined to at most 2 vertices in $V_i$.
If $v\in V_i$ is joined to exactly $2$ other vertices in $V_i$, then it has at most $3$ neighbours in one of the $V_j$'s and exactly $2$ neighbours in all the others.
So we can move one by one each vertex $v\in V_i$ with degree $2$ in $G[V_i]$ to $V_{(d-3)/2}$ or to $V_{(d-1)/2}$.
To obtain the final partition, we move the degree $2$ vertices $v\in G[V_{(d-3)/2}]$ to $V_{(d-1)/2}$, except for those with $3$ neighbours in $V_{(d-1)/2}$.
Finally, note that a vertex of degree $2$ in $G[V_{(d-1)/2}]$ is joined to at least $2$ vertices in each $V_i$ ($i\leq (d-3)/2$), hence is joined to at most $3$ vertices in  $V_{(d-3)/2}$.
\end{proof}

The following proposition states that paths and cycles can be realized on $\Sph^2$ in sufficiently general position. Note that when a $4$-cycle is realized on $\Sph^2$, there is always a pair of non-adjacent points that are diametrically opposite on the sphere.

\begin{proposition}\label{three} Any graph with maximum degree $2$ can be realized on $\Sph^2$ such that the following two properties hold:
\begin{enumerate}
\item For no $3$ distinct vertices $a$, $b$, and $c$, does there exist a vertex at distance $1$ from all three.
\item No $4$ vertices are on a circle, unless the $4$ vertices consist of two pairs of diametrically opposite points coming from two distinct $4$-cycles.
\end{enumerate}
\end{proposition}

In the proof we use ideas from the correction \cite{LSS00} to the paper \cite{LSS89} of Lov\'asz, Saks and Schrijver.
Let $G=(V,E)$ be a $(d-1)$-degenerate graph, and label its vertices as $V=\{v_1,\dots,v_n\}$ such that $|\setbuilder{v_j}{j<i\text{ and } v_iv_j\in E}|\leq d-1$ for all $i$.
We realize $G$ in $\Sph^{d-1}$ using a random process.
For any linear subspace $A$ of $\R^d$ of dimension at least $1$, there is a unique probability measure on the subsphere $A\cap \Sph^{d-1}$ that is invariant under orthogonal transformations of $A$, namely the Haar measure $\mu_A$.
Given the Haar measure $\mu$ on $\Sph^{d-1}$, $\mu_A$ on $A\cap\Sph^{d-1}$ can be obtained as the pushforward of $\mu$ by the normalized projection $\tilde{\pi}_A\colon\Sph^{d-1}\setminus A^\perp\to A\cap\Sph^{d-1}$ given by $\tilde{\pi}_A(x)=(\sqrt{2}|\pi_A(x)|)^{-1}\pi_A(x)$, where $\pi_A\colon\R^d\to A$ is the orthogonal projection onto $A$.

We now embed $G$ as follows.
We first choose $u_1$ distributed uniformly from $\Sph^{d-1}$ (that is, according to $\mu$).
Then for each $i=2,\dots,n$, we do the following sequentially.
Let $L_i = \Span\setbuilder{u_j}{j<i\text{ and }v_iv_j\in E}$, and choose $u_i$ uniformly from $L_i^\perp\cap\Sph^{d-1}$ (according to $\mu_{L_i^\perp}$) and independently of $\setbuilder{u_j}{j<i}$.

Since each $L_i$ has dimension at most $d-1$, this process is well defined. If $G$ has maximum degree at most $d-1$, then for any permutation $\sigma$ of $[n]$, the ordering $(v_{\sigma(1)},\dots,v_{\sigma(n)})$ has the property that $|\setbuilder{v_{\sigma(j)}}{j<i\text{ and } v_{\sigma(i)}v_{\sigma(j)}\in E}|\leq d-1$ for all $i$, and we can follow the above random process to embed $G$, thus obtaining a probability distribution $\nu_\sigma$ on the collection of realizations of $G$ in $\Sph^{d-1}$.
As pointed out in \cite{LSS00}, for different $\sigma$ we may obtain different probability distributions $\nu_\sigma$.
Nevertheless, as shown in \cite{LSS00}, under a certain condition on $G$, any two such measures are \emph{equivalent}, that is, they have the same sets of measure $0$, or equivalently, the same sets of measure $1$.
We say that an event $A$ holds \emph{almost surely} (\emph{a.s.}) with respect to some probability distribution if it holds with probability $1$.
\begin{lemma}[\cite{LSS00}]\label{ordering} For any graph $G=(V,E)$ that does not contain a complete bipartite graph on $d+1$ vertices, for any two permutations $\sigma$ and $\tau$ of $\{1,\dots,n\}$, the distributions $\nu_{\sigma}$ and $\nu_{\tau}$ are equivalent.
\end{lemma}
This lemma is used in \cite{LSS00} to show that under the same condition, the above random process gives a realization of the graph such that the points are in general position almost surely.
\begin{theorem}[\cite{LSS00, LSS89}]\label{realization}
For any graph $G=(V,E)$ that does not contain a complete bipartite graph on $d+1$ vertices, the above random process gives a realization of $G$ such that for any set of at most $d$ vertices of $G$, the embedded points are linearly independent.
\end{theorem}
We now apply Lemma~\ref{ordering} and Theorem~\ref{realization} to prove Proposition~\ref{three}.

\begin{proof}[Proof of Proposition~\ref{three}]
Observe that $G$ is a disjoint union of paths and cycles. If we remove a vertex from each $4$-cycle, we obtain a graph $G'=(V',E')$ with $V'=\{v_1,\dots,v_n\}\subseteq V$ that does not contain a complete bipartite graph on $4$ vertices (that is, a $4$-cycle or $K_{1,3}$). Take a random realization of $G'$ as described above, and then add back the removed vertices as follows. If $a$ was removed from the cycle $a v_i v_j v_k$ with this cyclic order, then embed $a$ as the point $-v_j$ opposite $v_j$.
We also denote $a$ by $-v_j$.
We claim that this realization satisfies the conditions of the proposition almost surely.

We want to avoid certain configurations on some small number of vertices. By Lemma~\ref{ordering} it is always enough to show that if we start with these few vertices then almost surely they do not form a prohibited configuration.

First we have to see that after adding back the removed vertices, we have a unit distance realization of $G$ almost surely. By Theorem~\ref{realization}, we have a realization of $G'$ almost surely, and for any $c$ with neighbours $b$ and $d$, we have that $b\neq \pm d$ a.s.\ and that no point is diametrically opposite $c$.
By adding back $a=-c$, we then also have $b$ and $d$ at distance $1$ from $a$.

Suppose next that some vertex $v$ is at distance $1$ to $a$, $b$, and~$c$.
If any of these vertices are in $V\setminus V'$, we may replace them by their diametrically opposite point which is in $V$, and we still have that $v$ is at distance $1$ to $a$, $b$, and $c$, and $v,a,b,c\in V'$.
Since $v$ is not adjacent to all three in $G'$, we may assume without loss of generality that $va\notin E'$.
If we then randomly embed $G'$ using an ordering that starts with $v$ and $a$, we obtain a.s.\ that $|v-a|\neq 1$, which is a contradiction (by Lemma~\ref{ordering}).
Therefore, no vertex of $G$ is at distance $1$ to three distinct vertices of~$G$.

We next show that no $4$ distinct vertices $w_1,w_2,w_3,w_4\in V$ of $G$ will be realized on a circle a.s., where $w_i=\epsi_i v_i$ for some $\epsi_i\in\{\pm 1\}$ and $v_i\in V'$, $i=1,2,3,4$, unless we have $w_1=-w_2$ and $w_3=-w_4$ after relabelling.
Suppose first that $v_1,v_2,v_3,v_4$ are distinct, and let $H:=G[v_1,v_2, v_3, v_4]$ and $H'=G[w_1,w_2,w_3,w_4]$.
Note that $v_i\mapsto w_i$ is an isomorphism from $H$ to $H'$.
Since $G$ does not contain a $4$-cycle or $K_{1,3}$, $d_H(v_i)\leq 1$ for some $1=1,2,3,4$.
Without loss of generality, $d_H(v_4)=d_{H'}(w_4)\leq 1$, and if $d_H(v_4)=1$, then $v_3v_4\in E'$.
Then $\dim L_4^\perp\geq 2$, and it follows that after choosing $u_3$, the fourth point $u_4$ and $-u_4$ will a.s.\ not be on the circle through $\epsi u_1$, $\epsi_2 u_2$, $\epsi_3 u_3$, since the great circle of $\Sph^2$ orthogonal to $u_3$ intersects each of the $8$ circles through any of $\pm u_1$, $\pm u_2$, $\pm u_3$ in at most $2$ points.

Next suppose that $v_1,v_2,v_3,v_4$ consist of exactly $3$ distinct vertices, say with $w_3=v_3=v_4$ and $w_4=-v_3=-v_4$.
Since $u_1$, $u_2$, $u_3$ are linearly independent a.s., none of the $8$ triples $\{\varepsilon_1 u_1, \varepsilon_2 u_2, \varepsilon_3 u_3\}$ where $(\varepsilon_1,\varepsilon_2,\varepsilon_3)\in\{\pm 1\}^3$, lie on a great circle a.s., hence $w_4$ is not on the circle through $w_1$, $w_2$, $w_3$ a.s.

The only remaining case is where $v_1,v_2,v_3,v_4$ consist of exactly $2$ distinct vertices, say with $w_1=-w_2=v_1$ and $w_3=-w_4=v_2$.
It follows that $w_1$ and $w_2$ are embedded as opposite points on $\Sph^2$, and $w_3$ and $w_4$ are too.
\end{proof}

\begin{proof}[Proof of Theorem~\ref{max degree d}]

For $d=1$ and $d=2$, the theorem is trivial.
For $d=3$, we use Proposition~\ref{three} as follows. First we remove vertices of degree $3$ in $G$ from $V$ one by one. Let $W\subset V$ be the set of removed vertices. Each $w\in W$ has exactly $3$ neighbours in $V$, $W$ is an independent set of $G$, and the maximum degree in $G[V\setminus W]$ is at most $2$. Now we represent $G[V\setminus W]$ on $\Sph^2$ as in Proposition~\ref{three}. Finally, we embed the removed vertices in $W$ one by one as follows. For any circle on $\Sph^2$, there are exactly $2$ points at distance $1$ from the circle. (They are not necessarily on the sphere.)
For any $w\in W$, we choose one of these two points determined by the circle through the $3$ neighbours of $w$.
It remains to show that there are at most $2$ vertices in $W$ that determine the same circle.
First note that at most $2$ vertices in $W$ can have the same set of neighbours, because $G$ does not have $K_{3,3}$ as a component.
Also, if $w_1\in W$ and $w_2\in W$ have different sets of neighbours, then their neighbours span different circles on $\Sph^2$.
Otherwise, if the neighbours of $w_1$ and $w_2$ lie on the same circle $C$, then by Proposition~\ref{three}, $w_1$ and $w_2$ have a common neighbour $v$ on $C$ that lies on a $4$-cycle in $G[V\setminus W]$, so $v$ will have degree $4$ in $G$, a contradiction.

For $d>3$ we consider two cases depending on the parity of $d$.

\smallskip
\textbf{Case 1:} $d$ is even. Let $V=V_1\cup\dots\cup V_{d/2}$ be a partition as in Lemma~\ref{more buses}. Remove vertices of degree $2$ in $G[V_{d/2}]$ from $V_{d/2}$ until the maximum degree of each remaining vertex in $V_{d/2}$ is at most $1$ in $G[V_{d/2}]$. Let $W\subset V_{d/2}$ be the set of removed vertices.
Then $W$ is an independent set of $G$, any $w\in W$ has exactly $2$ neighbours in $V_{d/2}$, and the maximum degree of a vertex in $G[V_{d/2}\setminus W]$ is at most $1$.
Hence $G[V\setminus W]=G[V_1\cup\dots\cup V_{(d/2)-1}\cup(V_{d/2}\setminus W)]$ can be represented on $\Sph^{d-1}$ as follows. As $G[V_i]$ for $1\leq i <d/2$ and $G[V_{d/2}\setminus W]$ have maximum degree $1$, they can be realized on circles of radius $1/\sqrt{2}$ and centre the origin $o$ in pairwise orthogonal $2$-dimensional subspaces of $\R^d$.
We can also ensure that no two vertices are diametrically opposite on a circle.

Then we add the vertices of $W$ one by one to this embedding. Each vertex $w\in W$ has exactly $2$ neighbours on each circle, so the set $N(w)$ of $d$ neighbours of $w$ span an affine hyperplane $H$ not passing through $o$, hence they lie on a subsphere of $\Sph^{d-1}$ of radius less than $1/\sqrt{2}$.
It follows that there are exactly $2$ points in $\R^d\setminus\Sph^{d-1}$ at distance $1$ from $N(w)$, both on the line through $o$ orthogonal to $H$.
We choose one of these points to embed $w$.

It remains to show that there are at most two $w\in W$ that determine the same subsphere, and that two different subspheres determine disjoint pairs of points at distance $1$.
There are no $3$ vertices in $W$ with the same set of neighbours, since the maximum degree in $V_{d/2}$ is at most $2$.
If some two vertices $w_1$ and $w_2$ from $W$ have different sets of neighbours $N(w_1)\ne N(w_2)$, then they have different pairs of neighbours on at least one of the orthogonal circles, so the affine hyperplanes $H_1$ and $H_2$ spanned by $N(w_1)$ and $N(w_2)$ are different.
If $H_1$ and $H_2$ are parallel, then the two subspheres $H_1\cap\Sph^{d-1}$ and $H_2\cap\Sph^{d-1}$ have different radii, and the pair of points at distance $1$ from $H_1\cap\Sph^{d-1}$ are disjoint from the pair of points at distance $1$ from $H_2\cap\Sph^{d-1}$.
If $H_1$ and $H_2$ are not parallel, the pairs of points at distance $1$ from $H_1\cap\Sph^{d-1}$ and from $H_2\cap\Sph^{d-1}$ lie on different lines through $o$ (and none can equal $o$), and so are also disjoint.
Therefore, all points from $W$ can be placed.

\smallskip
\textbf{Case 2:} $d$ is odd.
Let $V=V_1\cup \dots\cup V_{(d-3)/2}\cup V_{(d-1)/2}$ be a partition as in Lemma~\ref{more buses}.
First we embed $V\setminus V_{(d-3)/2}=V_1\cup \dots\cup V_{(d-5)/2}\cup V_{(d-1)/2}$ on $\mathbb{S}^{d-3}$ as follows.
As each $G[V_i]$ for $1\leq i \leq {(d-5)/2}$ has maximum degree $1$, the $G[V_i]$ can be realized on circles of radius $1/\sqrt{2}$ and with centre in the origin $o$ in pairwise orthogonal $2$-dimensional subspaces of $\R^d$.
We can also ensure that from $V_1\cup \dots\cup V_{(d-5)/2}$ no two vertices are diametrically opposite on a circle. Since the maximum degree of $G[V_{(d-1)/2}]$ is at most $2$, $V_{(d-1)/2}$ can be embedded on a $2$-sphere $S$ of radius $1/\sqrt{2}$ and centre $o$ in a subspace orthogonal to the subspace spanned by $=V_1\cup \dots\cup V_{(d-5)/2}$, as described in Proposition~\ref{three}.
We will denote by $C$ the circle of radius $1/\sqrt{2}$ and with centre $o$ in the plane orthogonal to the subspace spanned by $\mathbb{S}^{d-3}$.

Before treating the general case, we show that we can add $V_{(d-3)/2}$ to the embedding, assuming that $V_{(d-1)/2}$ is embedded in $S$ in general position in the sense that no four points of $V_{(d-1)/2}$ lie on the same circle and no three points of $V_{(d-1)/2}$ lie on a great circle of $S$.
With this assumption, embedding $V_{(d-3)/2}$ is very similar to the embedding of $V_{d/2}$ in the even case. First we find an independent set $W\subseteq V_{(d-3)/2}$ such that the maximum degree of $G[V_{(d-3)/2}\setminus W]$ is at most $1$, and each $w\in W$ has exactly two neighbours in $V_{(d-3)/2}$. Then we embed $V_{(d-3)/2}\setminus W$ on $C$ such that no two vertices are in opposite positions. Note that $V\setminus W$ is embedded in $\mathbb{S}^{d-1}$. Finally, we embed the vertices of $W$ one by one. Each vertex $w\in W$ has exactly two neighbours in $V_{i}$ for $1\leq i \leq {(d-3)/2}$ and three neighbours in $V_{(d-1)/2}$. By the general position assumption the affine hyperplane spanned by the set of neighbours $N(w)$ of $w$ does not contain the origin. Thus there are exactly $2$ points in $\mathbb{R}^d\setminus \mathbb{S}^{d-1}$ at distance $1$ from $N(w)$. We choose one of these points to embed $w$. An argument similar to the one that was used in the even case shows that there are at most two $w\in W$ that determine the same hyperplane, and two different hyperplanes determine disjoint pairs of points.

We now turn to the general case.
As before, we would like to choose an independent set $W\subseteq V_{(d-3)/2}$ such that the maximum degree of $G[V_{(d-3)/2}\setminus W]$ is at most $1$ and each $w\in W$ has exactly two neighbours in $V_{(d-3)/2}$.
However, this is not enough: Note that if $V_{(d-1)/2}$ is not in general position, then it is possible that there is a vertex $w\in V_{(d-3)/2}$ for which $N_1(w):=N(w)\cap V_{(d-1)/2}$ spans a great circle on $S$.  Hence the points that are at distance $1$ from $N(w)$ are the poles of the circle spanned by $N_1(w)$ on $S$. In addition, in this case the points that are at distance $1$ from $N(w)$ are determined by $N_1(w)$. Thus if for $w_1,w_2\in W$ we have $N(w_1)\ne N(w_2)$ but $N_1(w_1)$ and $N_1(w_2)$ span the same great circle on $S$, then the pair of points where $w_1$ and $w_2$ can be embedded, are the same.
Thus, we have to impose some more properties on the independent subset $W$.

Recall that $V_{(d-1)/2}$ is embedded on the $2$-sphere $S$ as in Proposition~\ref{three}.
Therefore, three vertices $a,b,c\in V_{(d-1)/2}$ can only span a great circle if two of them are opposite vertices of a $4$-cycle that are embedded in antipodal points. We assign an ordered triple $(a,b,c)$ to $a,b,c$ if they span a great circle with $a$ and $b$ being antipodal. By the properties of the embedding of $V_{(d-1)/2}$ on $S$, we have that $(a,b,c)$ and $(e,f,g)$ span the same great circle if and only if one of the following two statements hold.

\begin{enumerate}
\item  $\{a,b\}=\{e,f\}$, $c=g$, and no vertex from $V_{(d-1)/2}$ is embedded in the point antipodal to $c=g$. (That is, $c=g$ is not part of a pair of opposite vertices of a $4$-cycle that was embedded in an antipodal pair.)
\item $\{a,b,c,e,f,g\}=\{h,i,j,k\}$ consist of two pairs of points $\{h,i\}$ and $\{j,k\}$ that are opposite vertices of two $4$-cycles.
\end{enumerate}

If for $w_1,w_2\in V_{(d-3)/2}$, $N_1(w_1)$ and $N_1(w_2)$ are as in the first statement, they span the same great circle if and only if $N_1(w_1)=N_1(w_2)=\{a,b,c\}$. Since $a$ and $b$ have degree $2$ in $G[V_{(d-1)/2}]$, by Lemma~\ref{more buses} they are each joined to at most $3$ vertices in $V_{(d-3)/2}$, hence there are at most three vertices $w_1,w_2,w_3\in V_{(d-3)/2}$ for which $N_1(w_1)=N_1(w_2)=N_1(w_3)=\{a,b,c\}$. We will call such a triple $\{w_1,w_2,w_3\}$ a \emph{conflicting triple}.

If for $w_1,w_2 \in V_{(d-3)/2}$, $N_1(w_1)$ and $N_1(w_2)$ are as in the second statement, they span the same great circle in $S$ if and only if $N_1(w_1),N_1(w_2)\subseteq \{h,i,j,k\}$. Again, by Lemma~\ref{more buses}, any vertex from $\{h,i,j,k\}$ has at most three neighbours in $V_{(d-3)/2}$,  and so there are at most four vertices $w_1,w_2,w_3,w_4\in V_{(d-3)/2}$ for which $N_1(w_1),N_1(w_2),N_1(w_3),N_1(w_4)\subseteq \{h,i,j,k\}$. If there are $4$ such vertices we will call them a \emph{conflicting $4$-tuple}, while if there are $3$, we will also call them a \emph{conflicting triple}.

We will also call both a conflicting triple and a conflicting $4$-tuple a \emph{conflicting set}.
Note that any two different conflicting sets are disjoint.
Recall that by the properties of the embedding of $V_{(d-1)/2}$ given by Proposition~\ref{three}, if three vertices on $S$ span a great circle, no vertex from $V_{(d-1)/2}$ is embedded in the poles of this circle.
It follows that it is sufficient for an embedding to find $W\subseteq V_{(d-3)/2}$ with the following properties.

\begin{enumerate}
\item $W$ is an independent set.
\item If $w\in W$, then $w$ has exactly two neighbours in $V_{(d-3)/2}$ {\it(in order for $w$ to have exactly 3 neighbours in $V_{(d-1)/2}$).}
\item $V_{(d-3)/2}\setminus W$ can be embedded on $C$, such that if $a,b\in V_{(d-3)/2}\setminus W$ are neighbours of some $w\in W$, then $a$ and $b$ are not in opposite positions {\it(in order to guarantee that if for $w_1,w_2\in W$ the neighbour sets $N_1(w_1)$ and $N_1(w_2)$ span different circles, then $N(w_1)$ and $N(w_2)$ define different hyperplanes.)}
\item $W$ contains at most two points of any conflicting set \it{(in order to guarantee that the neighbours of at most two vertices from $W$ can define the same hyperplane)}.
\end{enumerate}

Once we find such $W$, we can proceed as in the particular case considered above. In the remaining part of the proof we construct such $W$.

Note that the connected components of $G[V_{(d-3)/2}]$ are paths and cycles.
We embed paths of length at most $3$ and cycles of length $4$ on $C$. Let $\mathcal{H}$ be the set of the remaining connected components of $G[V_{(d-3)/2}]$. It is easy to see the following:

\begin{proposition}\label{partition'}
Let $H\in \mathcal{H}$ be a cycle of length not equal to $4$ or a path of length at least $4$.
Then $V(H)$ can be partitioned into sets $A_H$ and $B_H$, so that:
\begin{enumerate}
\item $H[B_H]$ is a matching containing only vertices of degree $2$ in $H$ (that is, not containing endpoints of $P$).
\item For any maximal independent set $W'\subset B_H$ the graph $H[A_H\cup W']$ has connected components of size $\le 4$.
\end{enumerate}
\end{proposition}
\begin{proof}
Such a partition is very easy to achieve --- simply choose the edges in $B_H$ ``greedily'', in the path case starting from a vertex next to the endpoint of a path.
\end{proof}
For each $H\in\mathcal H$ we denote the partition given by Lemma~\ref{partition'} as $V(H) = A_H\cup B_H$, and select a maximal independent $W_H$ from each $B_H$, $H\in\mathcal H$, in a specific way to be explained below, and put $W:=\bigcup_{H\in\mathcal H} W_H$. First, let us verify that for any choice of $W$, we can make sure that the properties 1--3 are satisfied. First, clearly, $W$ is an independent set. Second, by the choice of $B$ in each component $H\in \mathcal H$, each vertex in $W$ has degree $2$.
Third, since each connected component of $H\setminus W$, $H\in \mathcal H'$, has size at most $4$, it can be realized on the circle $C$ such that vertices from different connected components are not in opposite position. Thus, if $w\in W$ has neighbours in different connected components of $H\setminus W$, then the property 3 is satisfied for $w$. If both neighbours of $w$ are in the same component of $H
\setminus W$, then $H$ is a cycle of length $3$ or $5$, and $H\setminus W = H\setminus \{w\}$ is a path of length $1$ or $3$. In both cases the neighbours of $w$ form an angle of $\pi/2$ and thus are not in opposite positions.

To conclude the proof, it remains to choose $W$ in such a way that property 4 is also satisfied.
Recall that $G[M]$ is a matching, where $M:= \bigcup_{H\in\mathcal H} B_H$, and $W\subset M$ has exactly $1$ vertex from each edge of $G[M]$. The vertices from $M$ may belong to several conflicting sets, but, since different conflicting sets are disjoint, each vertex belongs to at most one of them.

We add some new edges to $G[M]$ to obtain $G'$ as follows. For each conflicting triple, we add an edge between two of its vertices that were not connected before, and for each conflicting $4$-tuple we add two vertex disjoint edges that connect two-two of its vertices that were not connected before. It is clear that finding such edges is possible. Moreover, the added set of edges forms a matching. Thus, the graph $G'$ is a union
of two matchings, and therefore does not have odd cycles. Hence, $G'$ is bipartite, and it has an independent set $W$ which contains exactly one vertex from each edge in $G'$. This is the desired independent set, since no independent set in $G'$ intersects a conflicting group in more than two vertices.
\end{proof}

\section{Number of edges}

A graph $G=(V,E)$ is called \emph{$k$-degenerate} if any subgraph of $G$ has a vertex of degree at most $k$.

\begin{lemma}\label{degree d-2}
Let $d\geq 2$ and let $x$ be a vertex of degree at most $d-2$ in a graph $G$. If $G-x$ can be realized on $\Sph^{d-1}$ as a unit distance graph, then $G$ can also be represented on $\Sph^{d-1}$.
\end{lemma}

\begin{proof}
The neighbours of $x$ span a linear subspace of dimension at most $d-2$, so there is a great circle from which to choose $x$.
\end{proof}

\begin{corollary}\label{d-2-degenerate} Any $(d-2)$-degenerate graph has spherical dimension at most $d$.
\end{corollary}

The above corollary also follows from the proof of Proposition 2 in \cite{ES80}.

In the proof of Theorem~\ref{d+2 choose 2} we need the following well-known lemma.

\begin{lemma}\label{cross-polytope}  If the complement of a graph $H$ on $d+k$ vertices has a matching of size $k$, then $H$  can be realized on $\Sph^{d-1}$. In particular, the graph of the $d$-dimensional cross-polytope can be realized on $\Sph^{d-1}$.
\end{lemma}

\begin{proof}  Let $v_1,\dots, v_{d+k}$ be the vertices of $H$, labelled so that $v_{i}$ is not joined to $v_{d+i}$ ($i=1,\dots,d$).
Let vectors $e_1,e_2,\dots,e_d\in\Sph^{d-1}$ form an orthogonal basis.
Map $v_{i}$ to $e_{i}$ and $v_{d+i}$ to $-e_i$ ($i=1,\dots,d$).
This is the desired realization: $e_i$ is at distance $1$ from $\pm e_j$ whenever $j\ne i$.
\end{proof}

\begin{proof}[Proof of Theorem~\ref{d+2 choose 2}]
Define $g(2)=3$, $g(3)=8$ and $g(d)=\binom{d+2}{2}-1$ for $d\geq 4$.
We show by induction on $d\geq 2$ that if $G=(V,E)$ has at most $g(d)$ edges, then $G$ can be embedded in $\R^d$, and if $G$ furthermore does not contain $K_{d+1}$ or $K_{d+2}-K_3$, then $G$ can be embedded in the sphere $\Sph^{d-1}$ of radius $1/\sqrt{2}$.
This is easy to verify for $d=2$.
From now on, assume that $d\geq 3$, and that the statement is true for dimension $d-1$.

Remove vertices of degree at most $d-2$ one by one from $G$ until this is not possible anymore.
If nothing remains, Corollary~\ref{d-2-degenerate} gives that $G$ can be embedded in $\Sph^{d-1}$.
Thus, without loss of generality, a subgraph $H$ of minimum degree at least $d-1$ remains.
We first show that if $H$ contains $K_{d+1}$ or $K_{d+2}-K_3$, then $G$ can be embedded in $\R^d$, though not in $\Sph^{d-1}$.

Suppose that $H$ contains $K_{d+2}-K_3$.
Then $H$ cannot have more than $d+2$ vertices, since each vertex of $H$ has degree at least $d-1$.
Therefore, $H$ is contained in $K_{d+2}-e$, which can be embedded in $\R^d$ as two regular $d$-simplices with a common facet.
Note that this embedding has diameter $<2$.
There are at most two edges of $G$ that are not in $H$.
Then the degrees of the vertices in $V(G)\setminus V(H)$ are at most $2$, so they can easily be embedded in $\R^d$.

Suppose next that $H$ contains $K_{d+1}$ but not $K_{d+2}-K_3$.
If $H$ has more than one vertex outside $K_{d+1}$, then $|E(H)|\geq \binom{d+1}{2} + d-1 + d-2 > g(d)$, a contradiction.
If $H$ has a vertex outside $K_{d+1}$, then this vertex is joined to at least $d-1$ vertices of $K_{d+1}$, and it follows that $H$ contains $K_{d+2}-K_3$, a contradiction.
Therefore, $H=K_{d+1}$.
There are at most $g(d)-\binom{d+1}{2}\leq d$ edges between $V(H)$ and $V(G)\setminus V(H)$.
Therefore, some $v\in H$ is not joined to any vertex outside $H$.
%Choose a vertex $v$ of $H$.
Then $H-v=K_d$ can be embedded in $\Sph^{d-1}$, hence by Lemma~\ref{degree d-2}, $G-v$ can be embedded in $\Sph^{d-1}$.
Since $v$ is only joined to the $d$ vertices in $V(H-v)$, we can embed it in $\R^d\setminus\Sph^{d-1}$ so that it has distance $1$ to all its neighbours.

We may now assume that $H$ does not contain $K_{d+1}$ or $K_{d+2}-K_3$.
It will be sufficient to show in this case that $H$ can be embedded in $\Sph^{d-1}$, as it then follows by Lemma~\ref{degree d-2} that $G$ can also be embedded in $\Sph^{d-1}$.

If $H$ has at most $d+1$ vertices, then $H$ is a proper subgraph of $K_{d+1}$, hence $H$ is a subgraph of the cross-polytope, and we are done.

Suppose next that $H$ has $d+2$ vertices.
Then the complement $\overline{H}$ has maximum degree at most $2$.
If $\overline{H}$ has two independent edges, then $H$ is a subgraph of the $d$-dimensional cross-polytope.
Otherwise, the edges of $\overline{H}$ are contained in a $K_3$.
Then $H$ contains $K_{d+2}-K_3$, a contradiction.

Thus without loss of generality, $H$ has at least $d+3$ vertices.

Let $v$ be a vertex of maximum degree in $H$.
If $v$ is adjacent to all other vertices of $H$, then $v$ has degree at least $d+2$, hence $|E(H-v)|\leq g(d)-(d+2)\leq g(d-1)$, and, since $H$ does not contain  $K_{d+1}$ or $K_{d+2}-K_3$, the graph  $H-v$ does not contain $K_{d}$ or $K_{d+1}-K_3$. Therefore, by induction, $H-v$ is embeddable in $\Sph^{d-1}\cap H$, where $H$ is a hyperplane through the origin.
We then embed $v$ as a point on $\Sph^{d-1}$ orthogonal to $H$.

Thus without loss of generality, each vertex $v$ of maximum degree $\Delta$ has a non-neighbour $w$.
We may also assume that $\Delta\geq d$, otherwise Proposition~\ref{max degree d-1} gives that $H$ is embeddable in $\Sph^{d-1}$.
Then $|E(H-v-w)|\leq g(d)-\Delta-(d-1)\leq g(d)-d-(d-1)\leq g(d-1)$.
By induction, either $H-v-w$ is embeddable in $\Sph^{d-1}\cap H$, where $H$ is a hyperplane passing through the origin, and then $v$ and $w$ can be embedded as the two points on $\Sph^{d-1}$ orthogonal to $H$, or $H-v-w$ contains a $K_d$ or a $K_{d+1}-K_3$.

\smallskip
\textbf{Case 1:}  \emph{For any $v$ of maximum degree and any $w$ that is non-adjacent to $v$, $H-v-w$ contains a $d$-clique $K$.}
Since $H$ does not contain $K_{d+1}$, $v$ has a non-neighbour $x$ in $K$.
Then $H-v-x$ contains another $d$-clique $K'$.
If $K$ and $K'$ intersect in at most $d-2$ vertices, then $K\cup K'$ has at least $\binom{d+2}{2}-4$ edges, hence $|E(H)|\geq d+\binom{d+2}{2}-4 > g(d)$, a contradiction.
Therefore, $K$ and $K'$ intersect in exactly $d-1$ vertices, and $K\cup K'$ has at least $\binom{d+1}{2}-1$ edges.
Since $H$ has at least $d+3$ vertices, there exists a vertex $y\neq v$ not in $K\cup K'$.
Then $|E(H)|\geq \deg(v)+\deg(y)-1+\binom{d+1}{2}-1\geq d+(d-1)-1+\binom{d+1}{2}-1>g(d)$, a contradiction.

\smallskip
\textbf{Case 2:} \emph{Some vertex $v\in H$ of maximum degree $\Delta\geq d$ has a non-neighbour $w$ such that $H-v-w$ contains a $K_{d+1}-K_3$.}
Then $|E(H)|\geq\Delta+\deg(w)+\binom{d+1}{2}-3\geq g(d)$.
Since also $|E(H)\leq g(d)$, it follows that $H-v-w = K_{d+1}-K_3$, $v$ has degree $\Delta=d$, and $w$ has degree $d-1$.
Let $v_1,v_2,v_3$ be the pairwise non-adjacent vertices in $H-v-w$.
If $v$ is joined to at most $2$ of the $v_i$ and $w$ is joined to at most $1$ of the $v_i$, $i=1,2,3$, then the components of $H[v,w, v_1,v_2,v_3]$ are paths of length at most $3$, hence can be realized  on a great circle $C$ of $\Sph^{d-1}$, and the remaining $K_{d-2}$ can be realized on the subsphere orthogonal to $C$.
Otherwise, either $v$ is joined to all of $v_1, v_2, v_3$, or $w$ is joined to at least two of them.
Note that $v$ has a non-neighbour other than $w$ in $H$, and $w$ has at least $2$ non-neighbours other than $v$ in $H$.
It follows that there are two different vertices $w_1,w_2\in V(H)\setminus\{v,w\}$ such that $vw_1$ and $w w_2$ are non-adjacent pairs and $|\{w_1,w_2\}\cap \{v_1,v_2,v_3\}|\le 1$.
Thus, we can find three disjoint pairs of non-adjacent vertices in $H$ and apply Lemma~\ref{cross-polytope}.
\end{proof}

\section{Ramsey results}

To see that $f_{SD}(s)\ge \lceil (s+1)/2 \rceil$ consider the graph $G$ on $2d$ vertices which is a union of a $d+1$ clique and $d-1$ isolated vertices. Then $G$ contains $K_{d+1}$, so it cannot be embedded on $\mathbb{S}^{d-1}$  and $\overline{G}$ contains a copy of $K_{d+2}-K_3$.

To prove $f_{SD}(s)\le \lceil (s+1)/2 \rceil$ we show that if the edges of the complete graph on $2d-1$ vertices are coloured with red and blue, then either the graph spanned by the red (denoted by $G_r$) or the graph spanned by the blue edges (denoted by $G_b$) can be embedded on $\mathbb{S}^{d-1}$.

The proof is by induction on $d$. It is easy for $d=1,2$. For $d>2$: If the maximum degree of $G_r$ or $G_b$ is at most $d-1$, we are done by Proposition~\ref{max degree d-1}. So we may assume that there are two vertices, $v_r$ and $v_b$, of degree at least $d$ in $G_r$ and $G_b$ respectively. By the induction we may assume that $G_r[V-v_r-v_b]$ is realizable on $\mathbb{S}^{d-2}$. If the edge $v_rv_b$ is blue, we put $v_r$ and $v_b$ in the poles of the $(d-2)$-sphere on which $G_r[V-v_r-v_b]$ is embedded. Otherwise $v_b$ has at most $d-3$ neighbours in $G_r[V-v_r]$. In this case we first add $v_b$ on the $(d-2)$-sphere (on which $G_r[V-v_r-v_b]$ is embedded), and then we can put $v_r$ in one of the poles of the $(d-2)$-sphere.\\

To obtain the lower bound on $f_D(s)$ consider the graph $G$ which is the union of $K_{d+2}$ and $d$ isolated vertices. For this $G$ neither $G$ or $\overline{G}$ can be embedded in $\mathbb{R}^d$.

For the upper bound on $f_D(s)$ we show that if the edges of the graph on $2d$ vertices are coloured with red and blue, then either $G_r$ or $G_b$ can be embedded in $\mathbb{R}^d$. For any vertex $v\in V$ we have $d_{G_r}(v)+d_{G_b}(v)=2d-1$, so either $d_{G_r}\le d-1$ or $d_{G_b}\le d-1$. Hence we may assume that there are at most $d$ vertices that have degree larger than $d-1$ in $G_r$. Let $W$ be the set of vertices $v\in V$ with $d_{G_r}(v)\le d-1$. $|V\setminus W|\le d$, so we can embed $G_r(V)$ on $\mathbb{S}^{d-1}$. Then we add the vertices of $W$ to this embedding one by one as follows. If $w\in W$ has a neighbour in $W$, then it has at most $d-2$ neighbours in $V\setminus W$, thus we remove it from $W$ and embed it on $\mathbb{S}^{d-1}$. We repeat this until $W$ is an independent set. Now for each vertex $w\in W$ there is at least a circle (which is not necessarily contained in $\mathbb{S}^{d-1}$) in which we can embed $w$, so we embed them one by one.

\section{Additional questions}

In Theorem~\ref{max degree d} we proved that any graph with maximum degree $d$, except if $d=3$ and $G$ has $K_{3,3}$ as a component, can be embedded in $\mathbb{R}^d$. We suspect that a slightly stronger statement can be also true.

\begin{problem} Is it true that for $d>3$ any graph with maximum degree $d$, except $K_{d+1}$, has spherical dimension at most $d$?
\end{problem}

For $d=3$ the cube (and the cube minus a vertex) cannot be embedded on $\mathbb{S}^2$ neither the graphs on the vertices $a_1,\dots,a_n,b_1,\dots,b_n$ with the edge set $\setbuilder{a_ib_j}{j=i-1,i,i+1 \textrm{ mod } n}$ where $n\ge 3$ odd.

The lower and upper bound on $f_D(s)$ in Theorem~\ref{Ramsey} are very close, but still it does not give the exact value of $f_D(s)$. We conjecture that the lower bound is sharp.

\begin{problem} Is it true that either $G$ or $\overline{G}$ on $2d+1$ vertices has dimension at most $d$?
\end{problem}

\bibliographystyle{amsalpha}
\bibliography{biblio}

\end{document}